\documentclass[a4paper,12pt]{amsart}
\usepackage{CJK}
\usepackage{romannum}
\usepackage{amsfonts}
\usepackage{mathtools}
\usepackage{amsmath,amscd}

\usepackage{ifthen}
\usepackage{amsrefs}
\usepackage{mathrsfs}
\usepackage{amsthm}
\usepackage{amssymb}

\usepackage{tikz-cd}
  \usepackage{graphicx}
 \usepackage{relsize}

\usepackage{MnSymbol}
\usetikzlibrary{positioning, decorations.text}

\usepackage{hyperref}
\usepackage{url}
\usepackage{etoolbox}
\usepackage{rotating}

\usepackage[shortlabels]{enumitem}
\usepackage[paper=a4paper,left=20mm,right=20mm,top=25mm,bottom=30mm]{geometry}

\usepackage{tikz}
\usetikzlibrary{backgrounds}
\usepackage{tkz-euclide}
\usepackage{xcolor}
\usetikzlibrary{trees,snakes,shapes.geometric}
\usepackage[outline]{contour}
\contourlength{1.5pt}
\usetikzlibrary{positioning}
\usetikzlibrary{%
  matrix,%
  calc,%
  arrows%
}

\setlist[enumerate]{topsep=0em, itemsep= -0em, parsep = 0 em, label=$(\alph*)$}
\nocite{*}

\newcommand{\cK}{\mathcal{K}}

\DeclareMathOperator{\rep}{rep}

\DeclareMathOperator{\rk}{rk}

\DeclareMathOperator{\modd}{mod}

\DeclareMathOperator{\dimu}{\underline{dim}}

\newtheorem{proposition}{Proposition}[section]
\newtheorem{Theorem}[proposition]{Theorem}
\newtheorem{Lemma}[proposition]{Lemma}

\newtheorem{corollary}[proposition]{Corollary}
\newenvironment{example}[1][Example.]{\begin{trivlist}
\item[\hskip \labelsep {\bfseries #1}]}{\end{trivlist}}

\newenvironment{Remark}[1][Remark.]{\begin{trivlist}
\item[\hskip \labelsep {\bfseries #1}]}{\end{trivlist}}

\newenvironment{Definition}[1][Definition.]{\begin{trivlist}
\item[\hskip \labelsep {\bfseries #1}]}{\end{trivlist}}
\newenvironment{Acknowledgement}[1][Acknowledgement.]{\begin{trivlist}
\item[\hskip \labelsep {\bfseries #1}]}{\end{trivlist}}

\author{Jie Liu}

\title{Dimension vectors of elementary modules of generalized Kronecker quivers  }

\address{SUSTech International Center For Mathematics,   Southern University of Science and Technology, shenzhen 518055, China }
\email{hbecun@foxmail.com}
\begin{document}

\rmfamily

%%%%%%%%%%%%%%%%%%%%%% TITELSEITE %%%%%%%%%%%%%%%%%%%%%%%%%%%%%%%%

%\pagenumbering{roman}
\thispagestyle{empty}

\maketitle

\begin{abstract}
Let $k$ be an algebraically closed field. The generalized or $n$-Kronecker quiver $K(n)$ is the  quiver with two vertices, called a source and a sink, and $n$ arrows from source to sink. Given   a finite-dimensional module $M$ of the path algebra $kK(n)=\cK_n$,   we consider its dimension vector $\dimu   M=(\dim_k M_1, \dim_k M_2)$. Let  $\mathbf{F}=\{(x,y)\mid \frac{2}{n}x\leq y\leq x\}$, and let  $(x,y)\in\mathbf{F}$.   We  construct a module $X(x,y)$ of $\cK_n$, and we prove  it to be elementary. Suppose that $\dimu M=(x,y)$. We show that:
\begin{enumerate}
\item if $M$ is an elementary  module, then $x<2n$, and

\item when $x+y=n+1$, the module  $M$ is  elementary  if and only if  $M$ is of the form  $X(x,y)$.
\end{enumerate}
 
\end{abstract}

\section{introduction}
Let $k$ be an algebraically closed field,  and let   $\rep_k(K(n))$ denote the category of  finite-dimensional representations of $K(n)$. We denote the Bernstein-Gelfand-Ponomarev (BGP)  reflection functor by  $\sigma: \rep_k(K(n))\rightarrow\rep_k(K(n))$. It is well-known that $\sigma^2=\tau$, where $\tau$ is the Auslander-Reiten translation of $\rep_k(K(n))$ (cf. \cite{Gabriel}).   Given the path algebra $\cK_n$ of $K(n)$, we  use $\modd \cK_n$ to denote the category of finite-dimensional modules of $\cK_n$.  Since there exists an equivalence between the categories $\rep_k(K(n))$ and  $\modd\cK_n$, we usually use the terms  "representation" and "module" interchangeably. Let $M\in \rep_k(K(n))$ be indecomposable. We say that $M$ is \textit{regular}, provided $\sigma^t M\neq (0)$ for all $t\in \mathbb{Z}$.  Let $M\in \modd \cK_n$ be regular. Then a module $M$ is said to be \textit{elementary} if there is no short exact sequence $(0)\rightarrow L\rightarrow M\rightarrow N\rightarrow (0)$ with $L,N\in\modd \cK_n$ being non-zero regular
 modules.

   There is a quadratic form   $q(x,y)=x^2+y^2-nxy$  on the dimension vectors of generalized Kronecker modules.  We say that the dimension vector $(x,y)$ is \textit{regular}, provided $q(x,y)<0$. For   generalized  Kronecker quiver $K(n)$,  $n\geq 4$,
 \[
\begin{tikzcd}
    1 \circ
    \arrow[r, draw=none, "{\vdots}" description]
    \arrow[r, bend left,        "\gamma_1"]
    \arrow[r, bend right, swap, "\gamma_n"]
    &
    \circ 2,
\end{tikzcd}
\]
 not much is known about its elementary modules.  Let $\mathbf{R}$ be the set of regular dimension vectors. By abusing notations, we introduce two maps  $\sigma, \delta$ on the set $\mathbf{R}$,  where $\sigma(x,y)=(nx-y,x)$ and $\delta(x,y)=(y,x)$ for all $ (x,y)\in \mathbf{R}$.   Claus Michael Ringel gave a description of elementary modules (cf. \cite{Claus2}).  We now give a restriction on the dimension vectors of elementary modules.  Since Otto Kerner and Frank Lukas  prove that there are only finitely many   $(\sigma)^2$-orbits of dimension vectors of elementary modules of $\mathcal{K}_n$ (cf. \cite{Otto}),  it is possible to classify their  dimension vectors.

\begin{Acknowledgement}
The author wants to thank Rolf Farnsteiner, Daniel Bissinger, Hao Chang and Jan-Niclas Thiel for discussing this paper.

\end{Acknowledgement}

 \section{preliminaries}

 A finite-dimensional representation $M=(M_1,M_2, (M(\gamma_i))_{1\leq i\leq n})$ over $K(n)$ consists of vector spaces $M_{j},j\in\{1,2\}$,  and $k$-linear maps $M(\gamma_i)_{1\leq i\leq n}: M_{1} \rightarrow M_{2}$ such that $\dim_{k}M=\dim_k M_{1}+\dim_k M_2$ is finite. A morphism $f: M\rightarrow N$ between two representations of $K(n)$ is a pair $(f_1,f_2)$ of $k$-linear maps $f_j: M_1\rightarrow M_2$ $(j\in\{1,2\})$ such that for each arrow $\gamma_i: 1\rightarrow 2$,  there is a commutative diagram

\begin{center}
$\begin{array}[c]{ccc}
M_{1}&\stackrel{M(\gamma_i)}{\longrightarrow}&M_2\\
\downarrow\scriptstyle{f_1}&&\downarrow\scriptstyle{f_2}\\
N_1&\stackrel{N(\gamma_i)}{\longrightarrow}&N_2.
\end{array}$
\end{center}
Normally, we use  $S(i)$ to denote the simple representation and $P(i)$ (resp. $I(i)$) to denote the projective  (resp. injective) representation at the vertexes $i, i\in \{1,2\}.$    
 
  There is a function called dimension vector on $\modd \cK_n$
\begin{center}
\underline{dim}: $\modd \cK_n\rightarrow \mathbb{Z}^2, M \mapsto (\dim_k M_1, \dim_k M_2 )$.
\end{center}
If  $(0)\rightarrow L \rightarrow M \rightarrow N\rightarrow (0)$ is an exact sequence in $\modd \cK_n$, then  $\dimu L+\dimu N= \dimu M$.  We denote by $<-,->$    the  bilinear form \begin{center}
$<-,->: \mathbb{Z}^2 \times \mathbb{Z}^2 \rightarrow \mathbb{Z}, ((x_1,x_2),(y_1,y_2))\mapsto (x_1y_1+x_2y_2)-nx_1y_2.$
\end{center}
This bilinear form  coincides with the Euler-Ringel form on the Grothendieck group $K_0(\cK_n)\cong \mathbb{Z}^2$. Then we denote the corresponding quadratic form by
\begin{center}
$q:\mathbb{Z}^2\rightarrow \mathbb{Z},x\mapsto <x,x>$.

\end{center}

\begin{Definition}
A dimension vector $(x,y)$ is said to be regular,  provided $q(x,y)<0$. 

\end{Definition}

Let $\sigma, \sigma^-$ be the Bernstein-Gelfand-Ponomarev reflections (or BGP-functors) of $K_0(\cK_n)=\mathbb{Z}^2$ given by $\sigma(x,y)=(nx-y,x), \sigma^-(x,y)=(y,ny-x)$. Moreover,  we still use   $\sigma, \sigma^-$ to denote the BGP functors of $\modd \cK_n$ (we  take  the opposite of the $n$-Kronecker quiver to be again the $n$-Kronecker quiver). Let $M\in\modd \cK_n$ be an indecomposable module.  Then

\begin{enumerate}
\item[(1)]    $M$ is said to be  \textit{preinjective},  provided there exists  $t\in \mathbb{N}_0$ such that  $\sigma^{-t} M =(0)$.

\item[(2)]  $M$ is said to be \textit{preprojective},  provided there exists  $t\in \mathbb{N}_0$ such that $\sigma^{t} M=(0)$. 

\item[(3)] A module is said to be \textit{regular},  provided it does not contain indecomposable  direct summand which is preprojective or preinjective.

\end{enumerate}
If $M\in\modd\cK_n$ is an indecomposable module different from $S(2)$, then  $\dimu \sigma M=\sigma\dimu M$; similarly, if $M$ is indecomposable and different from $S(1)$, then we have $\dimu \sigma^-M=\sigma^- \dimu M$.  If module $M$ is an elementary module, then module $\sigma^t M$  is also an elementary module for all $t\in \mathbb{Z}$ \cite[VII. Corollary 5.7$(d)$]{Assem1}.

\begin{Definition}
The dimension vector $(x,y)$ is said to be \textit{elementary} (\textit{preprojective}, or \textit{preinjective}) provided there exists an  elementary  (preprojective, or preinjective)  module $M$ with $\dimu M=(x,y)$.
\end{Definition}

We now consider  the set   of regular dimension vectors  $\mathbf{R}$. We have seen that $\sigma$ maps $\mathbf{R}$ onto $\mathbf{R}$. In fact, there is another transformation $\delta$ on $K_0(\cK_n)$ defined by $\delta(x,y)=(y,x)$, and  it also sends $\mathbf{R}$ onto $\mathbf{R}$.  Let $M\in \modd \cK_n$.  Then $\delta (\dimu M)=\dimu M^*$,  where $M^*$ is the dual representation of $M$, that is,  $M^*=(M^*_1,M^*_2,(M^*(\gamma_i))_{1\leq i\leq n})$, $M^*_1$  is the $k$-dual of $M_2$ and $M^*_2$ is the $k$-dual of $M_1$, the map $M^*(\gamma_i)$ is the $k$-dual of $M(\gamma_i)$. We put

\begin{center}
$\mathbf{F}=\{(x,y)\mid \frac{2}{n}x\leq y\leq x\}.$
\end{center}

\begin{Lemma}\cite[Section 2. Lemma]{Claus2}
The set $\mathbf{F}$ is a fundamental domain for the action of the group generated by $\delta$ and $\sigma$ on the set $\mathbf{R}.$ 
\end{Lemma}

\begin{Lemma} \cite[Lemma 3.1]{Claus2}\label{regular}
Assume that $M\in\modd \mathcal{K}_n$ is a regular module with a proper non-zero submodule $U$ such that both dimension vectors $\dimu U$ and $\dimu M/ U$ are regular. Then $M$ is not elementary.
\end{Lemma}

\section{Dimension vectors of elementary modules }

By duality and Lemma \ref{regular},  a module $M\in$ mod $\cK_n$ is elementary if and only if its dual $M^*$ is elementary. That is, if $\dimu M=(x,y)$, then $(x,y)$ is elementary if and only if $(y,x)$ is elementary. Hence we only need to study one of these two dimension vectors.

\begin{Lemma}\cite[Lemma 14.11]{Daniel}\label{less n}
Let $M\in \modd \cK_n$  be an elementary module with $\dimu M=(x,y)$ and $y\leq x\leq y+n-2.$ Then $x< n.$
\end{Lemma}

Let  $V^i$ denote the $i$-dimensional vector space over the field $k,i\in \mathbb{N}$. Let $L(V^j,V^i)$ be the linear space consisting of all $k$-linear transformations from  $V^j$ to  $V^i$. Let $l(r,j,i)=\dim_k L'$, where  $L'  \subseteq L(V^j,V^i)$ is  the largest linear subspace such
that  $\rk v=r$ for all $0\neq v\in L'$.

\begin{Theorem} \cite[Theorem]{West}\label{wt}
Let $2\leq r\leq j\leq i$ be integers. Then 

\begin{center}
$i-r+1\leq l(r,j,i)\leq i+j-2r+1$.
\end{center}

\end{Theorem}

 We define a quadratic function $f(w):=w^2-nw+1=(w-\frac{n}{2})^2-\frac{n^2}{4}+1,w\in k,n\in \mathbb{N}$.  Since $w^2$ has a positive coefficient, we have $\underset{a\leq w\leq b}{\max}f(w)=\max \{f(a),f(b)\}$ for all $a\leq b \in \mathbb{R}$.
 
  Let $(x,y)\in \mathbf{F},y\neq n-1$. We have
 \begin{equation}
\begin{split}
q(x-1, y-(n-1)) & = (x-1)^2+(y-(n-1))^2-n(x-1)(y-(n-1)) \\ &  =(y-(n-1))^2((\frac{x-1}{y-(n-1)})^2+1-n\frac{x-1}{y-(n-1)}) \\& = (y-(n-1))^2(t^2+1-nt) \\& = (y-(n-1))^2 f(t), 
\end{split}
\end{equation}
where $t=\frac{x-1}{y-(n-1)}$. Then $q(x-1,y-(n-1))<0$ if and only if $f(t)<0$. 

Let $\Lambda_n$ be the  space spanned by arrows of $K(n)$. It is an $n$-dimensional vector space with basis $\{\gamma_i\mid 1\leq i\leq n\}$. Let $M\in$ mod $\cK_n$. We put  $\gamma_i.m=M(\gamma_i)(m)$ for all $m\in M$.  In general, we always assume that $n\geq 3$.

\begin{Lemma}\label{vector}
 Suppose that $(x,y)\in \mathbf{F}$ and $y\geq 2(n-1)$. Then $(x-1, y-(n-1))  $ is a regular dimension vector.

\end{Lemma}
\begin{proof}
Since $y\leq x$ and $n\geq 3$, we have $y-(n-1) \leq x-1.$ On the other hand, the inequalities $y\geq 2(n-1)$ and $y\geq\frac{2}{n}x$ imply  $y-(n-1)\geq\frac{n-1}{n(n-1)-1}(x-1).$  This is because $\frac{2}{n}>\frac{n-1}{n(n-1)-1}$ and these two lines $y_1=\frac{2}{n}x$ and $y_2=\frac{n-1}{n(n-1)-1}(x-1)+(n-1)$ intersect at the point $(n(n-1), 2(n-1))$.  Hence we have $ \frac{n-1}{n(n-1)-1}(x-1) \leq y-(n-1)\leq x-1$ and  $\frac{x-1}{y-(n-1)}\in [1, \frac{n(n-1)-1}{n-1}].$  Let $t=\frac{x-1}{y-(n-1)}$.  Since $f(1)\leq f(\frac{n(n-1)-1}{n-1})$, we obtain   

\begin{equation}
\begin{split}
f(t)=t^2+1-nt \leq\underset{1\leq t\leq \frac{n(n-1)-1}{n-1}}{\max} f(t) & = f(\frac{n(n-1)-1}{n-1}) \\ & = (\frac{n(n-1)-1}{n-1})^2+1-n\frac{n(n-1)-1}{n-1} \\ & =(n- \frac{1}{n-1})^2+1-n^2+\frac{n}{n-1} \\ &= n^2- 2-\frac{2}{n-1}+\frac{1}{(n-1)^2}+1-n^2+1+\frac{1}{n-1}\\ &=\frac{2-n}{(n-1)^2} <0.
\end{split}
\end{equation}

Hence $(x-1,y-(n-1))$ is a regular dimension vector. 
\end{proof}

\begin{Lemma}\label{fc}
Let $M\in\modd \cK_n$  be a  module such that $\dimu M=(x,y)$ and  $n-1\leq y\leq x+n-2$. Then $M$ has a submodule $U$ with dimension vector $(1,n-1)$.
\end{Lemma}
\begin{proof}

Actually, we can obtain a proof by following the arguments of Ringel mutatis mutandis (cf. \cite[Lemma 3.2]{Claus2}).

\end{proof}

\begin{Lemma}\label{y2n}
Suppose that $(x,y)\in \mathbf{F}$ is an elementary dimension vector. Then $y<2(n-1)$.

\end{Lemma}

\begin{proof}
By Lemma \ref{less n}, we only need to consider the case: $x-y\geq n-1$. Suppose that there exists an elementary module $M$ with  $\dimu M=(y,x)$ such that $y\geq 2(n-1)$ and $x-y\geq n-1,(x,y)\in \mathbf{F}$. Then we have 
\begin{equation}
\begin{cases}
y\geq 2(n-1),\\

x-y\geq n-1,\\

y \geq \frac{2}{n}x.

\end{cases}
\end{equation}

This implies 

\begin{equation}\label{xu}
\begin{cases}
y\geq 2(n-1),\\

  y+ n-1\leq x\leq \frac{n}{2}y,\\
  x\geq y+n-1 \geq 3(n-1).
\end{cases} 
\end{equation}

It follows that $y=y+(n-1)-(n-1)\leq x-(n-1)\leq \frac{n}{2}y-(n-1)$. Moreover, we have

\begin{center}
$\frac{y}{y-1}\leq \frac{x-(n-1)}{y-1}\leq \frac{\frac{n}{2}y-(n-1)}{y-1}$.

\end{center}
Since $f(\frac{y}{y-1})=-\frac{(n-2)y^2-(n-2)y-1}{(y-1)^2}<0$ and $f(\frac{\frac{n}{2}y-(n-1)}{y-1})=f(\frac{n}{2}-\frac{n-2}{2y-2})=(\frac{n-2}{2y-2})^2-\frac{n^2}{4}+1<0$, the dimension vector $(y-1,x-(n-1))$ is regular. Now we consider $(y-2,x-x'),n+1\leq x' \leq 2n-1$. Note that $y+n-1-x'\leq x-x'\leq \frac{n}{2}y-x'$ and $x'-(n-1)\leq n$. Then 
\begin{center}
$\frac{y+(n-1)-(2n-1)}{y-2}=\frac{y-n}{y-2}\leq  \frac{y+n-1-x'}{y-2}\leq \frac{x-x'}{y-2}\leq \frac{\frac{n}{2}y-x'}{y-2}\leq \frac{\frac{n}{2}y-(n+1)}{y-2}=\frac{n}{2}-\frac{1}{y-2}$. 

\end{center}
We have   $\frac{y-n}{y-2}=1-\frac{n-2}{y-2}\geq 1-\frac{n-2}{2(n-1)-2}=\frac{1}{2}$, and $f(\frac{1}{2})=\frac{5-2n}{4}<0,n\geq 3$. As $\frac{n}{2}-\frac{y-n}{y-2}\leq \frac{n}{2}-\frac{1}{2}$, we have $f(\frac{y-n}{y-2})\leq f(\frac{1}{2})<0$. On the other hand, $f(\frac{n}{2}-\frac{1}{y-2})=(\frac{1}{y-2})^2-\frac{n^2}{4}+1<0,y\geq 2(n-1)$. Hence $(y-2,x-x')$ is regular.

Let $U\subseteq M$ be the submodule generated by an arbitrary element $0\neq m\in M_1$. Then we have $\dimu U=(1,n)$.  Otherwise, suppose that $\dimu U=(1,x_U)$ and $x_U\leq n-1$. Since $M$ is indecomposable, we get $x_U\geq 1$, as $x_U=0$ will indicate that $M$ has $S(1)$ as a direct summand and $M$ is decomposable.  We now let $U''=U\oplus U'$, where $U'$ is a semi-simple module with dimension vector $(0,n-1-x_U)$. We can see that $U''\subseteq M$ is a  submodule of $M$ and $\dimu U''=(1,n-1)$. Since $(y-1,x-(n-1))$ is regular and \underline{dim} $U''$ is regular, it follows that   $M$ is not elementary, this is a contradiction to  Lemma \ref{regular}.

  We claim that for any two non-zero and linearly independent elements $m_1,m_2\in M_1$,  it generates a submodule $E\subseteq M$  with  $\dimu E=(2,x_{E})$ and $x_{E}\geq n+1$ . Otherwise, suppose that  there exist non-linear elements $m_1,m_2\in M_1$ such that the submodule $\tilde{E}$ generated by $m_1,m_2$ has dimension vector $(2,x_{\tilde{E}})$ and $x_{\tilde{E}}\leq n$.  We already know that for any submodule $U\subseteq M$ with dim$_kU_1\neq 0$, we have  $\dim_kU_2\geq n$. This implies $x_{\tilde{E}}= n$. Now Lemma \ref{fc} provides a submodule $E'\subseteq \tilde{E}\subseteq M$ of dimension vector $(1,n-1)$,  a contradiction. Hence  $x_{E}> n$.  We have seen that $(y-2,x-x')$ is regular for all $n+1\leq x'<2n$. Then $x_E=2n$. Note that $m_1,m_2$ are random elements of $M_1$.  Let  $\{e_1, \cdots,e_y\}$ be a basis of $M_1$. Let $U^i$ be the submodule generated by the single element  $e_i,i=1,\cdots,y$, where $y\geq 2(n-1)\geq 4$.  Comparing to the submodule $E$, we have  $U^i\cap U^j=(0)$ for any $i\neq j\in \{1,\cdots,y\}$. Then we can get $x\geq ny >\frac{n}{2}y$,  which is a contradiction to (\ref{xu}).  Finally, we can see that  there does not exist such elementary dimension $(y,x)$ with  $y\geq 2(n-1)$.  Hence $y<2(n-1)$.

\end{proof}

\begin{Lemma}\label{l2n}
Suppose that $(x,y)\in \mathbf{F}$ is an elementary dimension vector.  Then $x<2n$.
\end{Lemma}
\begin{proof}
By Lemma \ref{less n}, we only need to show that $(x,y)$ is not elementary when $x-y\geq n-1$, that is, $y\leq x-(n-1)$, where $x\geq 2n$. According to Lemma \ref{y2n}, we know that $\frac{2}{n}x\leq  y<2(n-1)$.  Now we consider the  dimension vector $(y,x)$.  Suppose that there exists an elementary module $M$ with  $\dimu M=(y,x)$, where $4=\frac{2}{2n}\times 2n\leq y\leq x-(n-1), x\geq 2n$.  We get 
\begin{equation}\label{iq}
\begin{cases}
2n\leq x\leq  \frac{n}{2}y ,\\

\frac{2}{n}x\leq y\leq x-(n-1),\\
y<2(n-1).
\end{cases} 
\end{equation}
Then $2n\leq x< n(n-1)$ and 
\begin{equation}
\begin{cases}
n+1\leq x-(n-1)< (n-1)^2, \\
\frac{2}{n}x-1\leq y-1,\\
y<2(n-1),\\
y\leq x-(n-1).
\end{cases}    
\end{equation}

We have

\begin{center}
 $\frac{y}{y-1}\leq \frac{x-(n-1)}{y-1}=t\leq \frac{x-(n-1)}{\frac{2}{n}x-1}=\frac{n}{2}(\frac{x-(n-1)}{x-\frac{n}{2}})=\frac{n}{2}(1-\frac{n-2}{2x-n})<\frac{n}{2}(1-\frac{n-2}{2n(n-1)-n})=\frac{n}{2}-\frac{n-2}{4n-6}$.
\end{center}
Note that $f(\frac{y}{y-1})=-\frac{(n-2)y^2-(n-2)y-1}{(y-1)^2}<0,4\leq y<2(n-1)$, and   $f(\frac{n}{2}-\frac{n-2}{4n-6})=(\frac{n-2}{4n-6})^2-\frac{n^2}{4}+1<1-\frac{n^2}{4}+1<0, n\geq 4$.  Then $f(t)<0,t\in [\frac{y}{y-1},\frac{n}{2}-\frac{n-2}{4n-6})$. Hence $(y-1,x-(n-1))$ is regular.

 Let $n+1\leq x'\leq 2n-1<x$.  Consider $(y-2,x-x')$. Then 
 \begin{center}
 $\frac{1}{n-1}=1-\frac{n-2}{2n-(n+1)}\leq 1-\frac{n-2}{x-(n+1)}=\frac{x-(2n-1)}{x-(n-1)-2}\leq \frac{x-(2n-1)}{y-2}\leq \frac{x-x'}{y-2}\leq \frac{x-(n+1)}{y-2}\leq \frac{x-(n+1)}{\frac{n}{2}x-2}$.
 \end{center}
Moreover,  $\frac{x-(n+1)}{\frac{n}{2}x-2}=\frac{n}{2}(\frac{x-(n+1)}{x-n})=\frac{n}{2}(1-\frac{1}{x-n})\leq \frac{n}{2}(1-\frac{1}{n(n-1)-n})=\frac{n}{2}-\frac{1}{2(n-2)},2n\leq x<n(n-1)$. 
On the other hand,  $f(\frac{1}{n-1})=\frac{2-n}{(n-1)^2}<0$ and $f(\frac{n}{2}-\frac{1}{2(n-2)})=(\frac{1}{2(n-2)})^2-\frac{n^2}{4}+1<0,n\geq 4$. Hence $(y-2,x-x')$ is regular. 

Let $U\subseteq M$ be a submodule with $\dimu U=(y_U,x_U)$. When $y_U=1$,  we get  $x_U=n$ since $(y-1,x-(n-1))$ is regular. Otherwise, suppose that $x_U\leq n-1$. We write $W=U\oplus U''$, where  $U''$ is a semi-simple module with  $\dimu U''=(0,n-1-x_U)$. Then  $\dimu W=(1,n-1)$ is regular and  $\dimu  M/W=(y-1,x-(n-1))$ is regular. Hence $M$ is not elementary by Lemma \ref{regular}. Suppose that $y_U=2$.  We claim that for any two non-zero and linearly independent elements $m_1,m_2\in U_1$,  it generates a submodule $U'\subseteq U$  with  $\dimu U'=(2,x_{U'})$ and $x_{U'}\geq n+1$. Otherwise, suppose that  there exist  elements $m_1,m_2\in U_1$ such that the submodule $\tilde{U}$ generated by $m_1,m_2$ has dimension vector $(2,x_{\tilde{U}})$  and $x_{\tilde{U}}\leq n$.  We  know that for any submodule $U\subseteq M$ with $\dim _kU_1\neq 0$,  we have $\dim_k U_2\geq n$. Hence $x_{\tilde{U}}= n$.  Now Lemma \ref{fc} provides a submodule $\bar{U}\subseteq \tilde{U}\subseteq M$ of dimension vector $(1,n-1)$, a contradiction.  We have seen that $(y-2,x-x_U)$  is regular when $n+1\leq x_U<x$. Moreover, $(2,x_U)$ is regular when $n+1\leq x_U<2n$. Hence $x_U=2n$.  This indicates that any two non-zero and linearly independent elements $m_1,m_2\in M_1$  generate a submodule $U'\subseteq M$ with  $\dimu U'=(2,2n)$. Let $\{e_1,\cdots,e_y\}$ be a basis of $M_1$. Let  $U^i$ be the submodule generated by the single element $e_i, i\in \{1,\cdots,y\}$. Then $U^i\cap U^j=(0)$.  Hence $x\geq ny >\frac{n}{2}y$, which is a contradiction to (\ref{iq}). Hence such an elementary module $M$ does not exist, this yields   $x<2n$.

 Suppose that we can find an elementary module $M$ with  $\dimu M=(y,x)$ when $x<2n$, where $(x,y)\in \mathbf{F}$. If we let $U$ be the submodule of $M$ generated by two non-zero and linearly independent elements $m_1,m_2\in M_1$,  then we have  $\dimu U=(2,x)$ according to the above discussion.

\end{proof}

\begin{corollary}\label{x2n}
Let $(x,y)\in \mathbf{F}$, $n<x<2n$. 
Let $M\in\modd \cK_n$ be a module with dimension vector $(y,x)$. Then $M$ is an elementary module if and only if the following two conditions hold.
\begin{enumerate}
\item Any non-zero element $m_1\in M_1$  generates a submodule $U'$ with dimension vector $(1,n)$, i.e.  $U'\cong P(1)$.

\item   Any two non-zero and linearly independent elements $m_1,m_2\in M_1$  generate a submodule $U''$ with dimension vector $(2,x)$. 

\end{enumerate}

\end{corollary}
\begin{proof}
Suppose that $M$ is an elementary module. According to the   proof of Lemma  \ref{l2n}, we get  $(a), (b)$.

Suppose that  $M_1$ satisfies conditions $(a), (b)$. We first show that  $M$ is indecomposable. Otherwise, suppose that $M=M'\oplus M''$.  Since $n<x<2n$,  there exists some submodule $X\in\{M',M''\}$ such that for any $0\neq m\in X_1$, it generates a submodule $U$ with dimension vector $(1,x_U)$ and $x_U<n$,  this contradicts $(a)$. Since $(x,y)\in \mathbf{F}$  and  $\dimu M=(y,x)$, $ y\geq \frac{2}{n}x>2$, module $M$ is regular. Finally,  $M$ is elementary by \cite[Appendix 1. Proposition]{Claus2}.

\end{proof}

\begin{Lemma}\label{2n}
Suppose that  $(x,y)\in \mathbf{F}$ is an elementary dimension vector, where $n < x<2n$. Then 
 $y^2-y+2x\leq 4n+2$.

\end{Lemma}
\begin{proof}
Suppose that $M$ is an elementary module with  $\dimu M=(y,x),2=\frac{2}{n}\times n< y \leq x$.  Let   $I=\{m_1,\cdots, m_y\}$ be a basis of $M_1$, and let  $I'=\{m'_1, \cdots, m'_x\}$ be a basis of $M_2$, respectively.  Let $m_i,m_j\in I,i\neq j$. For each $\gamma_i\in \Lambda_n$, there exist some $a^j_{is}\in k$ such that
\begin{center}
$\gamma_i.m_j=( a^j_{i1} , \cdots , a^j_{ix}),$
\end{center} 
$1\leq i\leq n, 1\leq s\leq x, j\in\{1, \cdots,y\} $.  We define a matrix:
\begin{center}

$A_{(i,j)}:= \begin{bmatrix}
\gamma_1.m_i \\
\vdots\\
\gamma_n.m_i\\
\gamma_1.m_j\\
\vdots\\
\gamma_n.m_j
\end{bmatrix}=\begin{bmatrix}
a^i_{11} & \cdots & a^i_{1x}\\ 
\vdots & \ddots & \vdots \\
a^i_{n1} & \cdots & a^i_{nx}\\
  a^j_{11} & \cdots & a^j_{1x}\\
 \vdots & \ddots & \vdots \\
 a^j_{n1} & \cdots & a^j_{nx}
\end{bmatrix}.$

\end{center}
It indicates that the rank of the matrix $A_{(i,j)}$ is independent of the choices of two elements $m_i,m_j$ by Corollay \ref{x2n}. Hence    $\rk A_{(i,j)} \equiv x$. On the other hand, the matrix $A_{(i,j)}$ can also be seen as  a linear transformation from the vector space $V^{x}$ to the vector space $V^{2n}$. Since the    cardinality of set $\{(m_i,m_j)\mid m_i, m_j\in I,i\neq j\} $ is $\binom{y}{2}=\frac{y(y-1)}{2}$,  the dimension of linear space $W$ generated by all such matrices $\{A_{(i,j)}\mid m_i\neq m_j\in I \}$ is  $\frac{y(y-1)}{2}$.  Otherwise, suppose that $\dim_k W<\frac{y(y-1)}{2}$. Then there exist some $ b_{i,j}\in k\setminus\{0\}$ such that $\sum_{i,j}b_{i,j} A_{(i,j)}=0$. That is, 
\begin{center}
$\sum_{i,j}b_{i,j} A_{(i,j)}=\begin{bmatrix}
\gamma_1.(\sum b_{i,j} m_i) \\
\vdots\\
\gamma_n.(\sum b_{i,j}m_i)\\
\gamma_1.(\sum b_{i,j}m_j)\\
\vdots\\
\gamma_n.(\sum b_{i,j} m_j)
\end{bmatrix}=\begin{bmatrix}
0 \\
\vdots \\
0 \\
\vdots\\

0
\end{bmatrix}.$
\end{center}
Then we can find  an element $0\neq m=\sum b_{i,j} m_i$ such that $\gamma_i.m = 0$ for all $1\leq i \leq n$, which  contradicts  Corollary \ref{x2n}$(a)$.

Finally, we have $2n-x+1\leq l(x,x,2n)\leq 2n+x-2x+1=2n-x+1,$ that is, $l(x,x,2n)=2n+1-x$. By Theorem \ref{wt}, we know that $ \frac{y(y-1)}{2}\leq 2n+1-x$, which means $y^2-y+2x\leq 4n+2$.

\end{proof}

\begin{Lemma}\label{de}
Let $M\in\modd \cK_n$ with dimension vector $(x,y),x\geq 1$. Suppose that  any non-zero element $m\in M_1$ generates a submodule $U$ with dimension vector $(1,y)$. Then $M$ is indecomposable and $y\leq n$.  
\end{Lemma}

\begin{proof}
When $x=1$, we are done. Now let  $x> 1$.    Suppose that $M$ is decomposable. We write $M=M'\oplus M''$. Assume that  $\dimu M'=(x',y')$. Any $0\neq m'\in M'_1$ generates a submodule $U'$ with dimension vector $(1,y)$.  We have $y'=y$. Hence $x'<x$, and $\dimu M''=(x-x',0)$. This cannot happen since an element $0 \neq m''\in M''_1 $ also generates a submodule $U''$ with dimension vector $(1,y)$.  Since $\dim _k \Lambda_n=n$,  we have $y\leq n$.

\end{proof}

\begin{Lemma}\label{mlessn}

Let $M\in\modd \mathcal{K}_n$ be a module with $\dimu M=(x,y)\in \mathbf{F}$ and $x \leq n-1.$ Let $0\neq m\in M_1$ be an arbitrary element.   Then $M$ is elementary if and only if the submodule $U$ generated by $m$ has dimension vector $(1,y)$.

\end{Lemma}

\begin{proof}
   Suppose that  $\dimu U=(1,y)$ for any $0\neq m\in M_1$.  Then  $U$ is a regular module by \cite[Lemma 3.5]{Daniel}. By Lemma \ref{de}  and  $1\leq y \leq x \leq n-1$, the module $M$ is regular indecomposable. Let $U'$ be a proper submodule of $M$ with $\dimu U'= (x',y')$. Then $x'\geq 1$. Note that $x'=0$ will indicate that $M$ is decomposable.  Hence $y'=y$ and   $\dimu U'$ is regular. Moreover,  $\dimu M/U'=(x-x', 0)$ is preinjective.  Hence $M$ is elementary by \cite[Appendix 1. Proposition]{Claus2}.

Conversely, let $M$ be an elementary module. Suppose that there exists some $0\neq m\in M_1$ generating a submodule $U$ with  $\dimu U=(1, y')$ and $1\leq y'<y.$  Note that $y'=0$ will indicate that $U$ is isomorphic to the simple module $S(1)$ and $M$ is decomposable. Then  $U$ is regular.  Now  consider  the factor module $M/U$ and its dimension vector $(x-1,y-y')$. We have
 \begin{equation}
 \begin{split}
q(x-1,y-y')& =(x-1)^2+(y-y')^2-n(x-1)(y-y')\\&  =(x-1)[(x-1)-n(y-y')]+(y-y')^2<0.
\end{split}
 \end{equation}
Since $x-1\geq y-y'$ and $n(y-y')-(x-1)>y-y'$ when $y\leq x \leq n-1,$  the dimension vector  $(x-1, y-y')$ is  regular,  which contradicts Lemma \ref{regular}.
\end{proof}

Let$(x,y)\in \mathbf{F}$, where  $x<n$.  Suppose that $M\in \modd \cK_n$ is an elementary module with dimension vector  $(x,y)$.  Lemma \ref{mlessn} tells us that $M$ can be seen as  a representation of the quiver

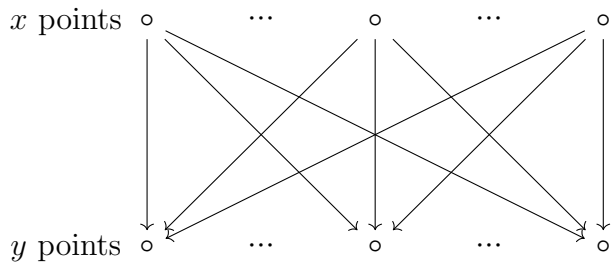
\begin{center}
\begin{tikzpicture}
\node (00) at (0,0) {$\circ$};
\node (08) at (-1,0) {$x \text{ points }$};
\node (25) at (1.5,0) {$\cdots$};

\node (30) at (3,0) {$\circ$};
\node (60) at (6,0) {$\circ$};

\node (35) at (4.5,0) {$\cdots$};

\node (0-3) at (0,-3) {$\circ$};
\node (0-9) at (-1,-3) {$y\text{ points }$};
\node (25) at (1.5,-3) {$\cdots$};
\node (3-3) at (3,-3) {$\circ$};

\node (35) at (4.5,-3) {$\cdots$};
\node (6-3) at (6,-3) {$\circ$};

\path [->] (00) edge (0-3)
           (00) edge(3-3)
           (00) edge(6-3)
           ;
\path[->]          (30) edge (0-3)      
         (30) edge  (3-3)                    
           (30) edge  (6-3)
                   ;
             
\path[->] (60) edge (0-3)  
   (60) edge (3-3)             
   (60) edge (6-3) ;                        
\end{tikzpicture}
\end{center}
where each  point of upper row has $y$  arrows. For the elementary module $M$, we just put the $1$-dimensional vector space $k$ at each point. We will give a precise construction later. 

Let $M\in\modd \cK_{n}$ be a module. We introduce an arrow $\gamma_{n+1}$ and define $\gamma_{n+1}.m:=0$ for each $m\in M$. Let $\Lambda_{n+1}=\Lambda_n\oplus k\gamma_{n+1}$, where $k\gamma_{n+1}$ is  1-dimensional vector space generated by the arrow $\gamma_{n+1}$.   Hence we have an embedding $\iota: \modd \cK_{n}\rightarrow \modd \cK_{n+1}$ in this way.  Then  $M$ can be seen as an object of  $\modd \cK_{n+1}$. In the following, we always mean $\gamma_{n+1}.M=0$ when we say $M\in \modd\cK_n\subseteq \modd \cK_{n+1}$.

Let $\mathbf{F}$ be the fundamental domain of $\cK_n$, and let $(x,y)\in \mathbf{F}$ with $x<n$. We have $\frac{2}{n}x\leq y\leq x$. Hence $\frac{2}{n+1}x<\frac{2}{n}x\leq y\leq x$. Then $(x,y)$ is located in the fundamental domain of $\cK_{n+1}$.

\begin{corollary}\label{coro}
Let $(x,y)\in \mathbf{F}$ with $x<n$. Suppose that $(x,y)$ is elementary for $\cK_{n}$. Then $(x,y)$ is elementary for $\cK_{n+1}$.

\end{corollary}
\begin{proof}
  Let $M\in\modd \cK_n$ be an elementary module
 with dimension vector $(x,y)$. Then $M$ is a module of $\modd\cK_{n+1}$ by the embedding $\iota$. We want to prove that $M$ is also elementary for $\cK_{n+1}$. Let $U$ be a submodule of $M$ in  $\modd \cK_{n+1}$. According to Lemma \ref{mlessn},   any $0\neq m_1\in M_1$  generates a submodule $U'$ in $\modd \cK_n$  with dimension vector $(1,y)$, and  this is also true for $\cK_{n+1}$. Then  $\dimu U=(x_U,y)$  for $\cK_{n+1}$, where $x_U\geq 1$. Then $M/U$ is preinjective with dimension vector $(x-x_U,0)$.  Hence we can see that $M$ is also elementary for $\cK_{n+1}$ by Lemma \ref{mlessn}.

\end{proof}

Let $(x,y)\in \mathbf{F}$ with $x<n$. Let $M\in \modd \cK_n$ be an elementary module  with  $\dimu M=(x,y)$.  Let $A$ be a matrix over $k$. We use $A^t$ to denote its transpose. Let $\{m_1,\cdots,m_x\}$ be a basis of $M_1$. We define $A_i:=\begin{bmatrix}
\gamma_1.m_i & \cdots & \gamma_n.m_i
\end{bmatrix}^t, \gamma_i\in \Lambda_n$. Lemma \ref{mlessn} tells us that  $\rk A_i=y$. 
We define $V_M:=k<A_{i}\mid  1\leq i\leq x>,$ that is, the linear space $V_M$ is spanned by all such $A_{i}$ over $k$. Then we have:

\begin{Lemma}\label{dimen}
Let $(x,y)\in \mathbf{F}$ with $x<n$.  Suppose that $M\in\modd \cK_n$ is an elementary module with  $\dimu M=(x,y)$. Then $\dim_kV_M=x$.

\end{Lemma}
\begin{proof}
Suppose that dim$_kV_M<x$. Then there exist some $b_i\in k\setminus\{0\}$ such that $\sum^x_{i=1}b_iA_{i}=0$. This means
\begin{center}
$ \sum^x_{i=1}b_iA_{i}=\begin{bmatrix}
\gamma_1.(\sum^x_{i=1}b_im_i) \\
\vdots \\
\gamma_x.(\sum^x_{i=1}b_im_i) \\
\vdots\\
 \gamma_n.(\sum^x_{i=1}b_im_i)
\end{bmatrix}=\begin{bmatrix}
0\\
\vdots\\
0\\
\vdots\\
0
\end{bmatrix}.  $
\end{center}
Hence we can find an element $0\neq m'=\sum^x_{i=1}b_im_i\in M_1$ such that $\gamma_j.m'=0$ for all $j\in \Lambda_{n}$. According to Lemma \ref{mlessn}, this cannot
 happen.

\end{proof}

\begin{Lemma}\label{limit}
Suppose that $(x,y)\in \mathbf{F}$ is an elementary dimension vector with $x<n$. Then $x+y\leq n+1$.
\end{Lemma}
\begin{proof}
Suppose that $M\in \modd \cK_n$ is an elementary module with  $\dimu M=(x,y)$. We now consider the linear space $V_M$.  Note that $A_{i}$ can be seen as a linear transformation from the vector space $V^y$ to the vector space $V^n$ and  $\rk A_i\equiv y$,   $A_i\in V_M$.  On the other hand,  $n-y+1\leq l(y,y,n)\leq n+y-2y+1=n-y+1$ by Theorem \ref{wt}.   Hence $\dim _k V_M=x\leq n-y+1,$ that is, $x+y\leq n+1$.

\end{proof}

\section{Construction of elementary modules}

Let  $(x,y)\in \mathbf{F}$ with $x+y=n+1$. We construct a  module $X=X(x,y)=(X_1,X_2,X(\gamma_i)_{1\leq i\leq n})$ of $\cK_n$: let $\{e_1,\cdots,e_x\}$ be a standard basis of $X_1$, i.e. the $i$-th coordinate is $1$ and all others are  $0$ in  each $e_i$. Let $\{e'_1,\cdots,e'_y\}$ be a standard basis of $X_2$.  We use $[s;a_1,a_2,\cdots,a_y]$ to denote  the arrow $\gamma_{a_i}$ mapping $e_s$ to $e'_{i}$ in $X(x,y)$, and call it \textit{arrow basis} of $e_s$, where $1\leq s\leq x, 1\leq  i \leq y, 1\leq a_i \leq n,  \gamma_{a_i}\in \Lambda_n$. For $e_1$, we define its arrow basis being $[1;1,2,\cdots,y]$. For the second $e_2$, we start from $2$ to  $y+1$, that is, the arrow basis of $e_2$ is $[2;2,3,\cdots,y+1]$. For $e_3$,  we repeat this process by starting from $3$ to   $y+2$.  We keep doing it. Finally, we can get an arrow basis of $X(x,y):[1;1,2,\cdots,y], [2;2,3, \cdots,y+1], \cdots, [s;s ,s+1,\cdots,y+s-1],\cdots,  [x;x,x+1,\cdots,n],1\leq s\leq x.$
That is, 

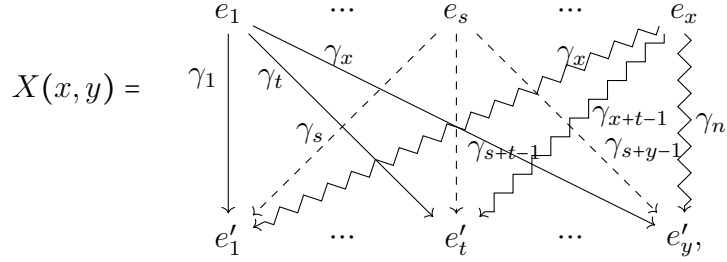
\begin{center}

\begin{tikzpicture}
\node (00) at (0,0) {$e_1$};
\node (25) at (1.5,0) {$\cdots$};
\node (30) at (3,0) {$e_s$};
\node (60) at (6,0) {$e_x$};

\node (35) at (4.5,0) {$\cdots$};
\node (22)  at (-2,-1){$X(x,y)=$};

\node (0-3) at (0,-3) {$e'_1$};
\node (25) at (1.5,-3) {$\cdots$};
\node (3-3) at (3,-3) {$e'_t$};
\node (35) at (4.5,-3) {$\cdots$};
\node (6-3) at (6,-3) {$e'_{y},$};

\path [->] (00) edge node[pos=.25,left]{$\gamma_1$} (0-3)
           (00) edge node[pos=.25, left]{$\gamma_t$} (3-3)
           (00) edge node[pos=.15, right]{$\gamma_{x}$}(6-3)
           ;
\path[dashed, ->]          (30) edge node[pos=.55, left]{$\gamma_{s}$} (0-3)      
         (30) edge node[pos=.65,right]{$\gamma_{s+t-1}$} (3-3)                    
           (30) edge node[pos=.65,right]{$\gamma_{s+y-1}$} (6-3)
                   ;
             
\draw[snake,segment length=10pt,->] (60)--(0-3)  node[pos=.15,left]{$\gamma_x$}; 
  \draw[snake,segment length=10pt,->] (60)--(3-3)  node[pos=.45,right]{$\gamma_{x+t-1}$};             
   \draw[snake,segment length=10pt,->] (60)--(6-3)  node[pos=.5,right]{$\gamma_{n}$};                        
\end{tikzpicture}

\end{center}
Let $e'_j=0$ when $j\nin \{1\,\cdots,y\}$. Then  $\gamma_i.e_j=e'_{i-j+1},i\in \Lambda_n,j\in \{1,\cdots,x\}$. Now we give an example. 

\begin{example}
Let $X(3,3)\in\modd \cK_5$. Then we can get its arrow basis: $[1;1,2,3],[2;2,3,4],[3;3,4,5]$. The structure of  $X(3,3)$ would be the following
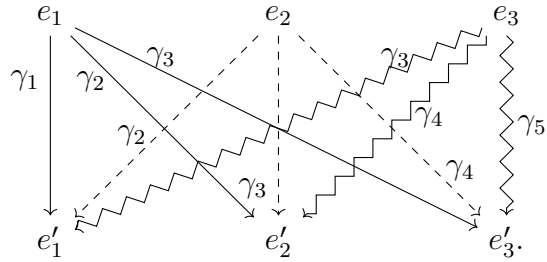
\begin{center}

\begin{tikzpicture}
\node (00) at (0,0) {$e_1$};

\node (30) at (3,0) {$e_2$};
\node (60) at (6,0) {$e_3$};

\node (0-3) at (0,-3) {$e'_1$};

\node (3-3) at (3,-3) {$e'_2$};

\node (6-3) at (6,-3) {$e'_{3}.$};

\path [->] (00) edge node[pos=.25,left]{$\gamma_1$} (0-3)
           (00) edge node[pos=.25, left]{$\gamma_2$} (3-3)
           (00) edge node[pos=.15, right]{$\gamma_{3}$}(6-3)
           ;
\path[dashed, ->]          (30) edge node[pos=.55, left]{$\gamma_{2}$} (0-3)      
         (30) edge node[pos=.85,left]{$\gamma_{3}$} (3-3)                    
           (30) edge node[pos=.75,right]{$\gamma_{4}$} (6-3)
                   ;
             
\draw[snake,segment length=10pt,->] (60)--(0-3)  node[pos=.15,left]{$\gamma_3$}; 
  \draw[snake,segment length=10pt,->] (60)--(3-3)  node[pos=.45,right]{$\gamma_{4}$};             
   \draw[snake,segment length=10pt,->] (60)--(6-3)  node[pos=.5,right]{$\gamma_{5}$};                        
\end{tikzpicture}

\end{center}
Then $X(3,3)=\begin{tikzcd}
   k^3 
    \arrow[r, draw=none, "{\vdots}" description]
    \arrow[r, bend left,        "X(\gamma_1)"]
    \arrow[r, bend right, swap, "X(\gamma_5)"]
    &
     k^3,
\end{tikzcd}$
where $X(\gamma_1)=\begin{bmatrix}
1 & 0 & 0\\
0 & 0 & 0\\
0 & 0 & 0 
\end{bmatrix}, X(\gamma_2)=\begin{bmatrix}
0 & 1& 0 \\
1 & 0 & 0\\
0 & 0 & 0
\end{bmatrix}, X(\gamma_3)= \begin{bmatrix}
0 & 0& 1\\
0 & 1& 0\\
1& 0 & 0
\end{bmatrix}, $

$X(\gamma_4)= \begin{bmatrix}
0 & 0& 0\\
0 & 0& 1\\
0& 1 & 0
\end{bmatrix}, X(\gamma_5)= \begin{bmatrix}
0 & 0& 0\\
0 & 0& 0\\
0& 0 & 1
\end{bmatrix}.$

\end{example}

\begin{Lemma}\label{X}
The module $X(x,y)$ is elementary.

\end{Lemma}

\begin{proof}
 Let $0\neq m=(a_1,\cdots,a_x)\in X_1$. Then we can get a matrix 
\begin{center}
$A_{m}=\begin{bmatrix}
\gamma_1.m \\
\vdots \\
\gamma_x.m \\
\vdots\\
 \gamma_n.m
\end{bmatrix}= \begin{bmatrix}
a_1 & 0 & 0 & \cdots & 0 \\
a_2 & a_1 &  0& \cdots &  0 \\
a_3 & a_2 & a_1 & \cdots  & 0 \\
 \vdots & \vdots & \ddots &\ddots & \vdots \\
 a_{y} & a_{y-1} & \cdots & a_2 & a_{1}\\
 a_{y+1} & a_y & \cdots & a_3 & a_2\\
\vdots  & \vdots &  & \vdots  & \vdots \\
a_x &  a_{x-1} & \cdots & a_{x-y+2}  & a_{x-y+1}\\
0 & a_{x} & \cdots  & a_{x-y+3} & a_{x-y+2}\\
\vdots  & \vdots & \ddots  & \ddots & \vdots \\
0 & 0  & \cdots  & a_{x} & a_{x-1}\\ 
 0 & 0& \cdots & 0  & a_x
\end{bmatrix}.$

\end{center}
Since $A_m$ is an $n\times y$ matrix,  we have  $\rk A_m\leq y$.  Let $j= \min \{i\mid a_i\neq 0, 1\leq i\leq x\}$. We get a  $y\times y$  matrix  $ A_j=\begin{bmatrix}
a_j & 0 & \cdots & 0 \\
a_{j+1} & a_j & \cdots   & 0\\
\vdots & \vdots & \ddots & \vdots \\
a_{j+y-1} & a_{j+y-2} & \cdots & a_j
\end{bmatrix}$.  Then  $\rk A_j=y$. Moreover,  we have a partition
\begin{center}
 $A_m=\begin{bmatrix}
A_j \\
A_0
\end{bmatrix}$,
\end{center}
where $A_0=\begin{bmatrix}
a_{j+y} & a_{j+y-1} & \cdots & a_{j+1}\\
a_{j+y+1} & a_{j+1} & \cdots & a_{j+2}\\
\vdots & \vdots & \ddots & \vdots \\
0 & 0& \cdots & a_x
\end{bmatrix}$. Then  $\rk A_m\geq y$ by \cite[8.2]{George}. Thus,  we get 
\begin{center}
rk $A_m\equiv y$. 
\end{center}
By Lemma \ref{mlessn},  $X(x,y)$ is elementary. 

\end{proof}

\begin{Theorem}
Let   $(x,y)\in \mathbf{F}$ with $x+y=n+1$. Then a module $M\in\modd \cK_n$ with dimension vector $(x,y)$ is elementary if and only if $M$ is of the form $X(x,y)$.

\end{Theorem}

\begin{proof}
By Lemma \ref{X}, we already know that $X(x,y)$ is elementary. Now  let $M\in \modd \cK_n$ be an elementary module with dimension vector  $(x,y)$. By Lemma \ref{dimen}, we get $\dim_k V_M=x$.  According to Theorem \ref{wt}, we have
\begin{center}
$n-y+1\leq l(y,y,n)\leq n+y-2y+1=n+1-y$.
\end{center}
That is, $x=n+1-y\leq l(y,y,n)=n+1-y$, which means $x=l(y,y,n)$. By Lemma \ref{X}, the  module $X(x,y)$ can also provide a linear space $V_X$ satisfying $\dim_k V_X=x$ and  $\rk v=y$ for all $0\neq v\in V_X$.  Then we have   $V_M\cong V_X$.  Suppose that $\{A_1,\cdots,A_x\}$ and $\{B_1,\cdots,B_x\}$ are two bases of the linear spaces $V_X$ and $V_M$, respectively. Then there exists an invertible matrix $T$ such that
\begin{equation}\label{INM}
(B_1,\cdots, B_x)^t=(A_{1},\cdots,A_{x})T, 
\end{equation}
 We now reconstruct a new module $X'=X'(x,y)$ such that the basis $\{A'_1,\cdots,A'_x\}$ of the linear space $V_{X'}$ satisfies $(A'_1,\cdots,A'_x)^t=(A_{1},\cdots, A_{x})T$. According to Lemma \ref{mlessn},  $X'(x,y)$ is elementary. Then we can see that $X'(x,y)\cong M$.

\end{proof}

\begin{Remark}
Unfortunately, we couldn't give a similar result for a dimension vector $(x,y)\in \mathbf{F}$ when $n<x<2n$. It is still a question to construct an    elementary module with such a dimension vector.

\end{Remark}

%%%%%%%%%%%%%%%%%%%%%%% REFERENCES %%%%%%%%%%%%%%%%%%%

\end{document}